\documentclass[11pt]{amsart}
\usepackage{pifont}
\usepackage{amsthm}
\usepackage{amsfonts}
\usepackage{amssymb}
\usepackage[mathscr]{euscript}
\usepackage[all]{xy}
\usepackage[dvips]{graphics}
\usepackage{amsmath}
\usepackage[dvips,final]{graphicx}
\usepackage[dvips]{geometry}
\usepackage{color}
\usepackage{epsfig}
\usepackage{latexsym}
\usepackage{subcaption}
\usepackage{tikz} 
\graphicspath{{diagrams/}}

\newtheorem{theorem}{Theorem}

\newtheorem{corollary}[theorem]{Corollary}

\theoremstyle{definition}

\newtheorem{example}{Example}[subsection]

\theoremstyle{remark}

\newcommand{\mse}{\mathsf{e}}
\newcommand{\mso}{\mathsf{o}}
\newcommand{\mseb}{\bar{\mathsf{e}}}
\newcommand{\msob}{\bar{\mathsf{o}}}
\newcommand{\mcob}{\bar{\mathcal{O}}}
\newcommand{\mceb}{\bar{\mathcal{E}}}

\newcommand{\mce}{\mathcal{E}}
\newcommand{\mco}{\mathcal{O}}

\title{Virtual Knot Groups}
\author{Heather A. Dye}
\author{Aaron Kaestner}

\begin{document}

\begin{abstract}For a knot diagram $K$, the classical knot group $\pi_1(K)$ is a free group modulo relations determined by Wirtinger-type relations on the classical crossings. The classical knot group is invariant under the Reidemeister moves. In this  paper, we define a set of quotient groups associated to a knot diagram $K$. These quotient groups are invariant under the Reidemeister moves and the set includes the extended knot groups defined by Boden et al  and Silver and Williams.  \end{abstract}
\maketitle
\section{Introduction}

An \textit{oriented virtual knot diagram} is a decorated immersion of an oriented $S^1$ into the plane with two types of crossings. Classical crossings are indicated by over/under markings and virtual crossings are circled. An example of a virtual knot diagram is shown in figure \ref{fig:labeledvknot}. 

\begin{figure}
\[ \scalebox{0.75}{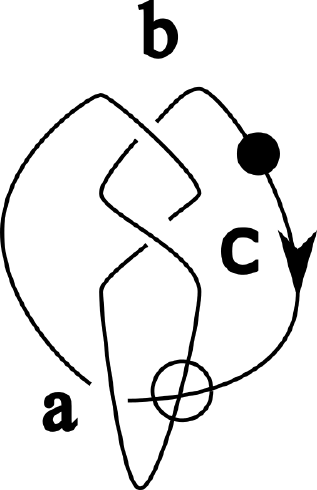} \]
\caption{Labeled, oriented knot}
\label{fig:labeledvknot}
\end{figure}

An \textit{oriented virtual knot} is an equivalence class of virtual knot diagrams determined by  oriented Reidemeister moves (figures \ref{fig:or1move}, \ref{fig:or2move} and \ref{fig:or3move}) and an oriented detour move as in figure \ref{fig:detour}. An oriented detour moves deletes an oriented arc of the knot diagram that contains only  virtual crossings and redraws the arc in a new location preserving orientation. Any new double points are virtual crossings. 
\begin{figure}
\[ \begin{array}{c} \scalebox{0.5}{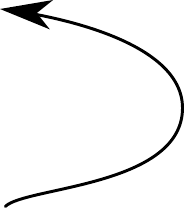} \end{array} \leftrightarrow
\begin{array}{c} \scalebox{0.5}{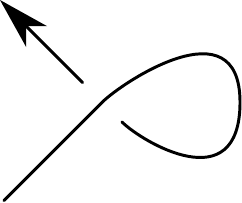} \end{array}
\leftrightarrow
\begin{array}{c} \scalebox{0.5}{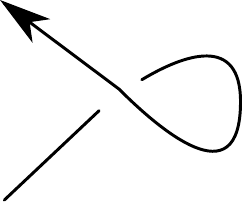} \end{array}
 \]
\caption{Oriented Reidemeister I} 
\label{fig:or1move}
\end{figure} 

\begin{figure}
\begin{subfigure}{0.48\linewidth}
\[ \begin{array}{c} \scalebox{0.5}{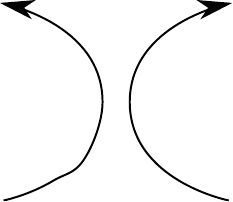} \end{array} \leftrightarrow
\begin{array}{c} \scalebox{0.5}{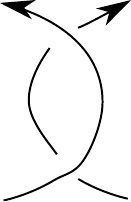} \end{array} \]
\caption{Oriented Reidemeister II a}
\label{fig:or2movea}
\end{subfigure} 
\begin{subfigure}{0.48\linewidth}
\[ \begin{array}{c} \scalebox{0.5}{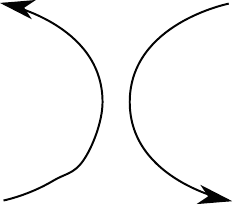} \end{array} \leftrightarrow
\begin{array}{c} \scalebox{0.5}{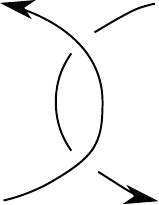} \end{array} \]
\caption{Contra-oriented Reidemeister II b}
\label{fig:or2moveb}
\end{subfigure}
\caption{Oriented Reidemeister II}
\label{fig:or2move}
\end{figure}

\begin{figure}
\[ \begin{array}{c} \scalebox{0.5}{\input{diagrams/or3alhs.pdf_tex}} \end{array} \leftrightarrow
\begin{array}{c} \scalebox{0.5}{ \input{diagrams/or3arhs.pdf_tex}} \end{array} \]
\caption{Oriented Reidemeister III}
\label{fig:or3move}
\end{figure}

\begin{figure}
\[ \begin{array}{c} \scalebox{0.5}{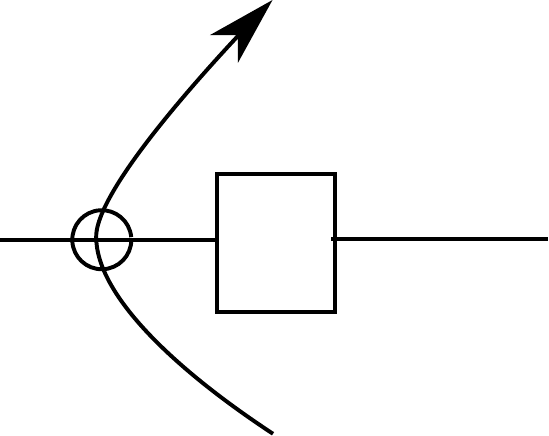} \end{array} \leftrightarrow
\begin{array}{c} \scalebox{0.5}{ 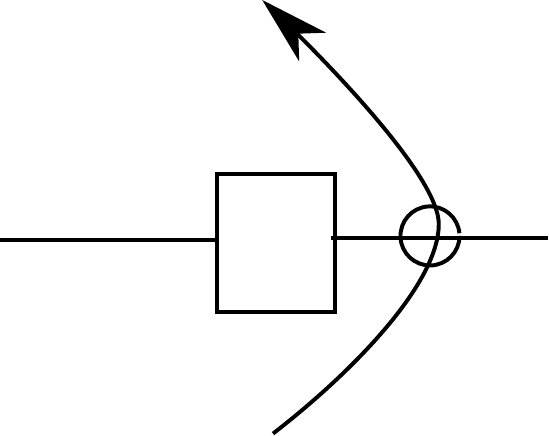} \end{array} \]
\caption{Oriented detour move}
\label{fig:detour}
\end{figure}

With orientation,  a minimal  generating set of Reidemeister moves consists of two Reidemeister I moves (figure \ref{fig:or1move}), two Reidemeister II moves (figure \ref{fig:or2move}) and a single Reidemeister III move (figure \ref{fig:or3move}) as shown by Polyak \cite{polyak} and the oriented detour move (\ref{fig:detour}).
The oriented detour move can be decomposed into oriented versions of the virtual Reidemeister moves shown in figure \ref{fig:vrmoves}. The necessary orientations are analogous to the classical orientations.

\begin{figure}
\begin{subfigure}{0.49\linewidth}
\[ \begin{array}{c} \scalebox{0.5}{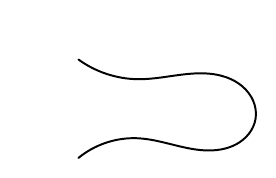} \end{array} \leftrightarrow
\begin{array}{c} \scalebox{0.5}{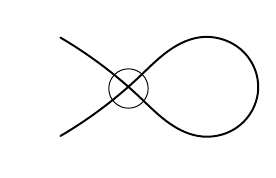} \end{array} \]
\caption{Virtual I}
\label{fig:vr1move}
\end{subfigure}
\begin{subfigure}{0.49\linewidth}
\[ \begin{array}{c} \scalebox{0.5}{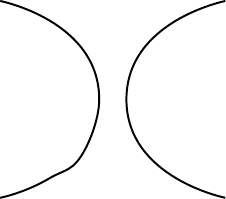} \end{array} \leftrightarrow
\begin{array}{c} \scalebox{0.5}{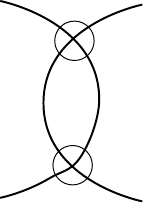} \end{array} \]
\caption{Virtual II}
\label{fig:vr2move}
\end{subfigure} \\
\begin{subfigure}{0.49\linewidth}
\[ \begin{array}{c} \scalebox{0.5}{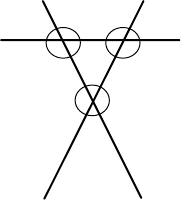} \end{array} \leftrightarrow
\begin{array}{c} \scalebox{0.5}{ 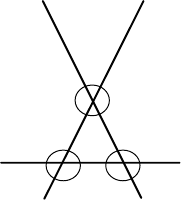} \end{array} \]
\caption{Virtual III}
\label{fig:vr3move}
\end{subfigure}
\begin{subfigure}{0.49\linewidth}
\[ \begin{array}{c} \scalebox{0.5}{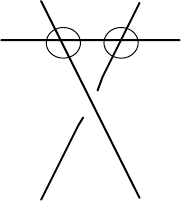} \end{array} \leftrightarrow
\begin{array}{c} \scalebox{0.5}{ 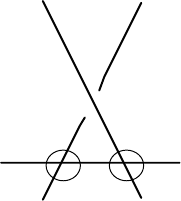} \end{array} \]
\caption{Virtual IV}
\label{fig:vr4move}
\end{subfigure}
\caption{Virtual Reidemeister moves}
\label{fig:vrmoves}
\end{figure}
The Gauss code of a knot diagram records the order of the crossings. 
To find the Gauss code of the diagram, we select a base point and traverse the knot in the direction of orientation and record the label of each crossing as it is traversed. In figure \ref{fig:labeledvknot}, the Gauss code is
\[ abcacb. \]
The \textit{parity of a classical crossing} in the knot is determined by whether  there are an even or odd number of labels between the two occurences of the crossing label in the Gauss code \cite{introvkt}, \cite{dkencyc}. The crossing has \textit{even parity} if there is an even number of terms between the two labels. Otherwise, the crossing has \textit{odd parity}.  In figure \ref{fig:labeledvknot}, the crossing $a$ has even parity  and the crossings $b$ and $c$ are oddly intersticed and have odd parity.

In a diagram with no virtual crossings, all classical crossings have even parity.
The parity of a crossing is unchanged by Reidemeister moves and virtual Reidemeister moves that do not act on that crossing. Only Reidemeister I and II moves introduce or remove new classical crossings. A Reidemeister I move introduces a single crossing with even parity. The Reidemeister II move  introduces either two odd crossings or two even crossings with opposite signs. The Reidemeister III move does not change the parity of any crossing in the diagram.

For more information about parity, refer to works on parity by Manturov et al \cite{ManturovIlyutkoNikonov}.

\subsection{Knot Groups} \label{knotgroups}

In virtual knot theory, researchers have considered a variety of different groups associated with virtual knots \cite{bodengroup}. These groups are quotients of a free group obtained by associating a generator with each arc or semi-arc of diagram and a set of Wirtinger-type relations based on the crossings.

\begin{figure}
\begin{subfigure}{0.32\linewidth} 
\[ \scalebox{0.5}{ 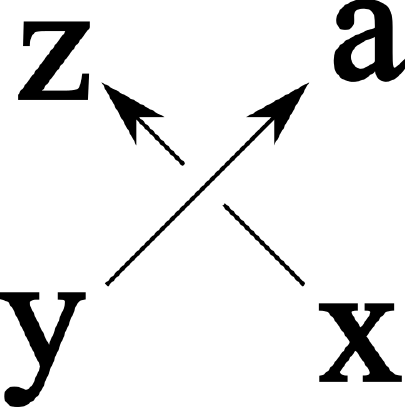}  \]
\caption{Positive crossing}
\label{fig:positivecrossing}
\end{subfigure}
\begin{subfigure}{0.32\linewidth} 
\[ \scalebox{0.5}{ 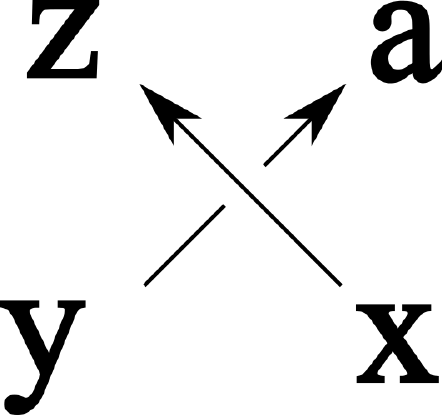}  \]
\caption{Negative crossing}
\label{fig:negativecrossing}
\end{subfigure}
\begin{subfigure}{0.32\linewidth} 
\[ \scalebox{0.5}{ 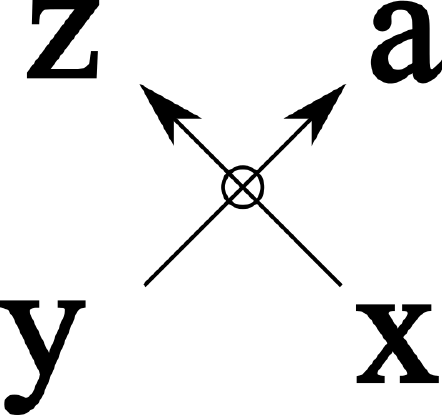}  \]
\caption{Labeled virtual crossing}
\label{fig:virtualcrossing}
\end{subfigure}
\caption{Labeled signed crossings}
\label{fig:labeledcrossings}
\end{figure}

In  the paper,\textit{Virtual knot groups and almost classical knots} \cite{bodengroup}, the virtual knot group, $VG(K)$, is defined using this method. Positive crossings (as labeled in figure \ref{fig:labeledcrossings}) define the 
relations: $z=y^{-1} x s y s^{-1}$ and $a=sys^{-1}$.  Virtual crossings define the relations
$a = qyq^{-1}$ and $z=qxq^{-1}$.  Specializations of this group define many of the groups studied in virtual knot theory.
The quandle group, $QG(K)$, is obtained by letting $s=1$ \cite{introvkt}. If $q=s$, the welded group, $WG(K)$, (defined by Fenn and Rourke \cite{fennrourke}) is obtained. The extended group $EG(K)$  \cite{silverwilliamsgroup} is obtained by letting $q=1$. Finally, the fundamental group of a classical knot is obtained by 
letting $s=q=1$.

Previously, the authors of this paper incorporated parity information in the knot group structure in \textit{Virtual Parity Alexander Polynomial} \cite{dyekaestner}. This resulted in a knot group where even positive crossings determined the relations $z=y^{-1} x s y s^{-1}$ and $a=sys^{-1}$.  The odd crossings give rise to the relations
$a = \theta y \theta^{-1}$ and $z=\theta x \theta^{-1}$. The virtual crossings define the relations
$a = qyq^{-1}$ and $z=qxq^{-1}$. This quotient group contains parity information, but at the cost of the information about the sign of the odd crossings. 

In section \ref{autgroups}, we define a set of quotient groups associated to a virtual knot. In section \ref{structure}, we discuss the structure of this set of knot groups and its relationship to previously studied knot groups.   In section \ref{paritygroups}, we use this structure to determine if signed parity information can be incorporated into the group structure.  

\section{Knot groups defined via elements of $Aut(F_{2n+j})$} \label{autgroups}
Given a virtual knot diagram $K$, we construct two different quotients of the free group $F_{2n+j}$ that are invariant under the oriented Reidemeister moves and the detour move. 

Let $n$ be the number of classical crossings in $K$ and choose $j \geq 0$. Then $K$
 has $2n$ semi-arcs that begin and end at classical crossings. Each semi-arc is associated with one of the generators $a_1, a_2, \ldots , a_{2n}$ of $F_{2n+j}$. Then choose two automorphisms of 
$F_{2n+j}$, $\phi$ and $\theta$. We  quotient $F_{2n+j}$ by the  commutator relations determined by $\theta$ and $\phi$ and relations determined by the set of classical crossings. In the second group structure, we let $n$ equal the number of classical and virtual crossings. We then choose three automorphisms of the free group $F_{2n+j}$. We quotient $F_{2n+j}$ by commutator relations defined by the automorphisms and relations determined by the classical and virtual crossings. 

Within this set of quotient groups, we recover 
the knot groups described in \cite{bodengroup}. Within the set of quotient groups defined, we are able to obtain a variety of algebraic structures that have been motivated by knot diagrams and the Reidemeister moves.

\begin{theorem}\label{novirts}
Let $K$ be an virtual oriented knot diagram with $n$ classical crossings. Let $j \geq 0$ be an integer and let $ \phi$ and $ \theta $ be automorphisms of $F_{2n+j}$. 
The quotient group $(K, F_{2n+j}, \theta, \phi) $ is determined by the set of $2n$ relations determined by the classical crossings and the automorphisms as follows. (See figure \ref{fig:labeledcrossings}  for crossing labels.)
 From positive crossings, we obtain the relations
\begin{align*} 
z&= y \theta(x a^{-1}), & a= \phi (y).
\end{align*}
From negative crossings, 
\begin{align*} 
z&= \phi^{-1} (x), & a = \theta^{-1} (z^{-1} y ) x .
\end{align*}
We denote these relations as $r_1, r_2, \ldots r_{2n}$. 
Let 
$ R= \lbrace r_1, r_2, \ldots r_{2n} \rbrace \cup \lbrace \phi(\theta(a)) = \theta(\phi(a)) : a \in F_{2n+j}\rbrace .$
Then 
$(K, F_{2n+j}, \phi, \theta) = F_{2n+j}/R $ is  invariant under the oriented Reidemeister moves and detour move.
\end{theorem}

\begin{proof}

Consider a labeled positive crossing as shown in figure \ref{fig:positivecrossing}
\begin{align*}
z&= y \theta(x a^{-1}), & a= \phi (y).
\end{align*}
Using the co-oriented Reidemeister II move (figure \ref{fig:or2movea}), we calculate that
negative crossings must define
\begin{align*}
z&= \phi^{-1} (x), & a = \theta^{-1} (z^{-1} y ) x .
\end{align*}
The contra-oriented Reidemeister II move (figure \ref{fig:or2moveb}) is also invariant under the relations.

We prove invariance under the oriented Reidemeister move III as shown in figure \ref{fig:R3a}. 

\begin{figure}
\[ \begin{array}{c} \scalebox{0.5}{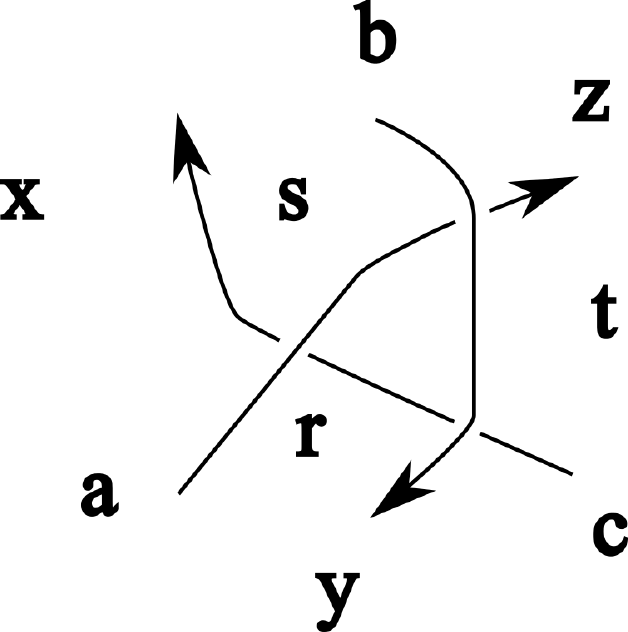} \end{array} \leftrightarrow \begin{array}{c} \scalebox{0.5}{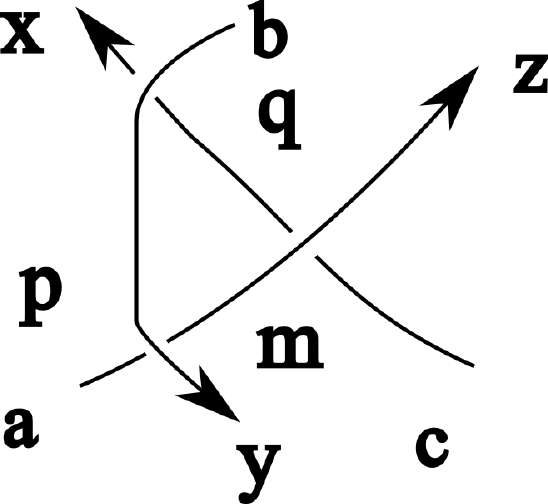} \end{array} \]
\caption{Labeled Reidemeister III move} 
\label{fig:R3a}
\end{figure} 

From the left hand side in figure \ref{fig:R3a}, we obtain:
\begin{align*}
x&=a \theta (r s^{-1}), & s &= \phi (a),  \\  
z& = b \theta (s t^{-1} ), & t &= \phi (b), \\  
r&= \theta^{-1} ( y^{-1} c) t, & y &= \phi^{-1} (t) .  
\end{align*}
These equations reduce to 
\begin{align} \label{eqn3alhs}
x & = a b ^{-1} c \theta ( \phi ( b a^{-1})), & y&=b, & z &=b \theta ( \phi (ab^{-1})).
\end{align} 

From the right hand side, we obtain: 
\begin{align*} 
m &= p \theta (a y^{-1}), & \phi (p) &=y \\
q&= m \theta (c z^{-1}), & \phi (m) &= z \\  
x &= \theta^{-1} (p^{-1} q) b,  & \phi^{-1} (b) &= p  
\end{align*} 
This set of equations reduces to
\begin{align} \label{eqn3arhs}
x & = a b ^{-1} c \phi ( \theta ( b a^{-1})), & y&=b, & z &=b \phi ( \theta (ab^{-1})).
\end{align}
Comparing equations \ref{eqn3alhs} and \ref{eqn3arhs}, 
\begin{align*} \label{eqn3alhs} 
a b ^{-1} c \theta ( \phi ( b a^{-1})) &=  a b ^{-1} c \phi ( \theta ( b a^{-1})), & b \theta ( \phi (ab^{-1}))&= b \phi ( \theta (ab^{-1})).
\end{align*}
Since the automorphisms $\phi$ and $\theta$ commute, then the relations are invariant under the generating set of the oriented Reidmeister moves.
\end{proof}

Next, we incorporate the virtual crossings into the structure of the quotient group. In this corollary, the endpoints of a semi-arc are at any crossing, regardless of type.

\begin{corollary}\label{withvirts} Let $K$ be an oriented knot diagram with $n$ total crossings (both classical and virtual). Choose $j \geq 0$, an integer. 
Let $\eta, \phi $, and $ \theta$  be automorphisms of $F_{2n+j}$.  The quotient group $(K, F_{2n+j},\theta, \phi, \eta)$
is obtained by quotienting $F_{2n+j}$ by relations obtained from crossings and automorphisms. 
 For positive crossings
\begin{align*}
z&= y \theta(x a^{-1}), & a= \phi (y).
\end{align*}
For negative crossings, 
\begin{align*}
z&= \phi^{-1} (x), & a = \theta^{-1} (z^{-1} y ) x .
\end{align*}
For virtual crossings,  
\begin{align*} a&= \eta(y), & z&= \eta^{-1} (x).
\end{align*}  
See figure \ref{fig:labeledcrossings} for crossing labels.
We obtain the set of relations:
$$ R = \lbrace r_1, r_2, \ldots r_{2n} \rbrace \cup \lbrace
\phi (\theta (a) ) = \theta (\phi (a)),  \phi ( \eta (a) = \eta ( \phi (a)),  \eta ( \theta (a)) = \theta ( \eta (a)) : a \in F_{2n+j} \rbrace
$$

Then the quotient group $ (K, F_{2n+j}, \eta, \theta, \phi) = F_{2n+j}/R$ is  invariant under the oriented Reidemeister moves and detour move.
\end{corollary}

\begin{proof}
We apply Theorem \ref{novirts} to obtain invariance under the oriented Reidemeister moves.  
For the labeled virtual crossing, we let $a= \eta(y)$ and $z= \eta^{-1} (x)$.  Direct computations verify invariance under the virtual Reidemeister I, II, and III moves. We focus on the virtual Reidemeister IV move (figure \ref{fig:R4a}). 
\begin{figure}
\[ \begin{array}{c} \scalebox{0.5}{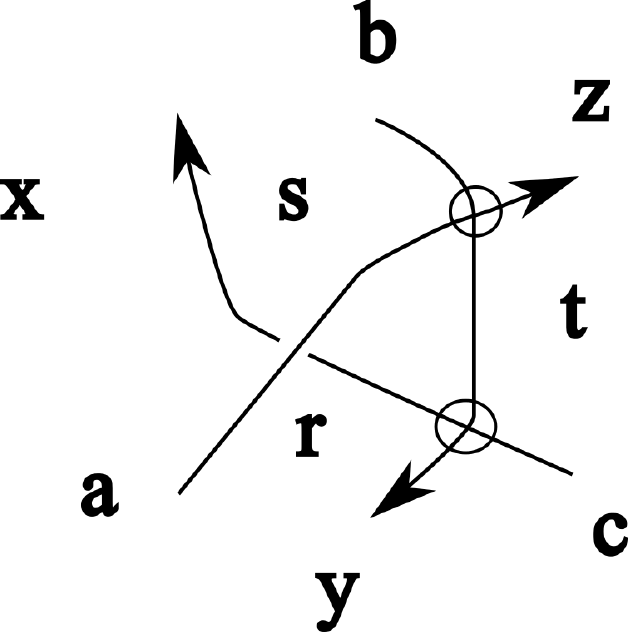} \end{array} \leftrightarrow \begin{array}{c} \scalebox{0.5}{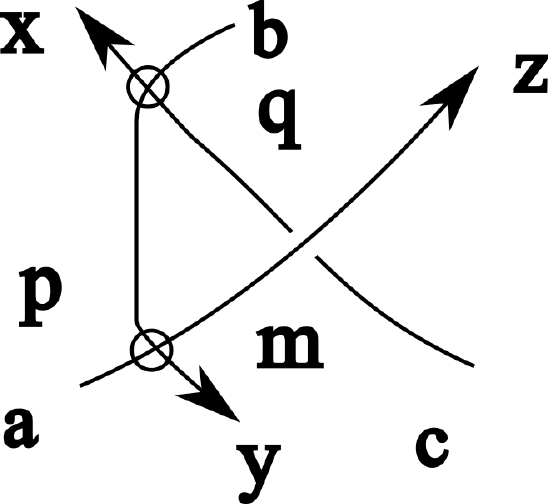} \end{array} \]
\caption{Oriented Virtual Reidemeister IV move} 
\label{fig:R4a}
\end{figure} 

From the left hand side of figure \ref{fig:R4a}, we obtain
\begin{align*} \label{eqnvr4lhs}
x&=a \theta ( r s^{-1} ), & s&= \phi (a), \\
y& = \eta^{-1} (t), & r &= \eta(c),  \\
z&= \eta^{-1} (s), & t &= \eta (b).  
\end{align*}
These equations reduce to 
\begin{align*}
x &= a  ( \theta (\eta (c) \phi (a^{-1})))  , & b&= y,  & z &=  \eta^{-1}( \phi (a)). 
\end{align*}

From the right hand side, we obtain
\begin{align*}
q &= m \theta (c z^{-1}),  & z &= \phi (m), \\
m &= \eta^{-1} (a), & y &= \eta (p). \\
p &= \eta^{-1} (b), & x &= \eta (q).
\end{align*}
These equations reduce to 
\begin{align*}
x &= a \eta ( \theta (c \phi ( \eta^{-1}( a^{-1})   ) )) ,& b&= y,   & z &= \phi ( \eta^{-1} (a)) . 
\end{align*}
Since the automorphisms $ \eta, \theta$, and $\phi$ commute then the relations define a quotient group that is invariant under the  oriented classical and virtual Reidemeister moves. 
\end{proof} 
\subsection{Examples}

We focus on specific examples of Theorem \ref{withvirts}. 
Let $K$ be a virtual knot diagram with $n$ total crossings. 
\begin{example}
Let  $ \phi = \eta = \theta = Id$ and let $j=0$. Then  $(K, F_{2n}, Id, Id, Id)$ defines a Wirtinger-type presentation group. For classical knots, this is $\pi_1 (K)$.
\end{example}

\begin{example}
Let  $ \phi = \eta = Id$, then $ (K, F_{2n}, Id, \theta, Id)$ defines a quandle for a fixed $\theta$. In this case, $z=y \theta (xy^{-1}) $ and $y=a$.   In \cite{joyce}, 
Joyce defined a quandle operation on a group $G$ by selecting and automorphism of $G$ and defining the quandle operation
$x \rhd y = s(x y^{-1})y$. Building on this definition, for a subgroup $H$ of $G$,  $G/H$ inherits a quandle structure: $Hx \rhd Hy = Hs(xy^{-1})y$. This is a homogeneous quandle (for any pair of elements, $a$ and $b$, there is a quandle homomorphism sending $a$ to $b$). 

Further, Joyce proved that
\begin{theorem}[\cite{joyce}, Theorem 7.1] Every homogenous quandle has a representation of the form $(G/H, z)$ where $z$ defines an inner automorphism. \end{theorem}
and
\begin{theorem}[\cite{joyce}, Theorem 7.2 ] Every quandle is representable as $(G/H_1, H_2 \ldots, z_1, z_2, \ldots).$ \end{theorem} 
\end{example} 

\begin{example}
Let  $ \theta= Id$, then $ (K, F_{2n}, Id, Id, \phi)$ defines a virtual biquandle for a fixed $\phi$. In this case, 
$z=y x \phi(y^{-1}) $ and  $a = \phi (y)$. 
We recall the virtual biquandle definition given in \cite{cranshenrichnelson}.
Let $X$ be a set and let $ \Delta : X \rightarrow X \times X$, be defined as $\Delta(x) = (x,x)$. Then let $B,V: X \times X \rightarrow X \times X$ denote invertible maps satisfying four axioms:
\begin{enumerate}
\item $V^2 = Id: X \times X \rightarrow X \times X,$
\item There exist unique invertible maps $S, vS: X \times X \rightarrow X \times X$ satisfying
\begin{equation} \nonumber
S(B_1(y,x),y) = (B_2(y,x),y) \text{ and } vS(V_1 (y,x),y)=(V_2(y,x),x)
\end{equation}
for all $x,y \in X$,
\item $(S^{\pm 1} \circ \Delta)_j$  and $(vS^{\pm 1} \circ \Delta)_j$ satisfying 
$(S \circ \Delta)_1 = (S \circ \Delta)_2$ and $(vS \circ \Delta)_1 = (vS \circ \Delta)_2$
\item $B$ and $V$ satisfy the set theoretic Yang Baxter equations
\begin{align*}
(B \times Id) (Id \times B) (B \times Id) = (Id \times B)(B \times Id) (Id \times B)\\
(V \times Id) (Id \times B) (V \times Id) = (Id \times V)(B \times Id) (Id \times V)\\
(V \times Id) (Id \times V) (V \times Id) = (Id \times V)(V \times Id) (Id \times V)
\end{align*}
\end{enumerate} 

We note that our proof of Theorem \ref{novirts} and Corollary \ref{withvirts} demonstrates axiom 4. 
We define $S$ and $vS$ as
\begin{align*}
S( a, z)=(\phi (z), z^{-1} a \phi(z)) \text{ and } vS(a,z)=(\eta (z), \eta(a)).
\end{align*}
Then, $(S \circ \Delta)_1 = (S \circ \Delta)_2$. 
Our construction satisfies all the axioms of the virtual biquandle. 
\end{example}

\begin{example}
Choose $j=3$. Consider $F_{2n+3}$ with generators
$a_1, a_2 \ldots a_{2n}, s,t,q$.  Let $\phi, \eta, \theta$ be elements of the inner automorphism group of $F_{2n+3}$ determined by $r,s,t$. We recover the extended virtual knot group, $VG(K)$ from Boden et al \cite{bodengroup} and discussed in section \ref{knotgroups}. 
\end{example} 

\section{The Structure of Theorem \ref{novirts}} \label{structure} 

We consider the set of groups defined by Theorem \ref{novirts}. In figure \ref{fig:groupstruct}, we diagram the possible algebraic structures specialized from $(K,F_{2n+j}, \theta, \phi )$. 
Let $V(K)$ denote the set of knot structures determined by  pairs of automorphisms of $F_{2n+j}$. Let $I$ denote the identity automorphism on $F_{2n+j} $. 
If $ \theta = I$, then we recover $B(K)$, a set of groups with a biquandle structures. 
If $ \phi = I$, then we recover $Q(K)$, a set of groups with quandle structures. The structures $B(K)$ and $Q(K)$ have been studied in a variety of papers. Further, letting $\phi= \theta$, we obtain $S(K)$ a set of group structures where $z= y \theta ( x \theta(y^{-1}))$ and $a= \theta (y)$. If $ \phi = \theta^{-1}$ then we obtain the set of group structures, $I(K)$ where $z = y \theta (x) y^{-1}$ and $ a= \theta^{-1} (y)$. The structures in $I(K)$ and $S(K)$ do not appear to have been studies in the literature.  Finally, letting $\theta=\phi =I$, we obtain the classical fundamental group $\pi_1 (K)$. This follows the specialization structure given in Boden et al \cite{bodengroup} on a larger scale. 
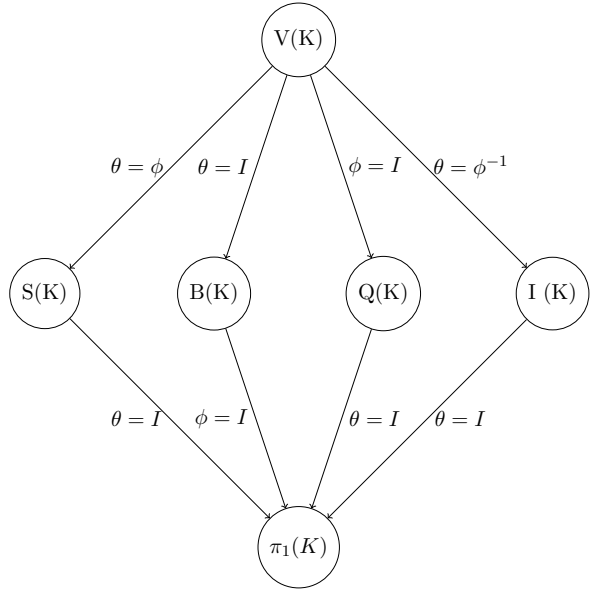
\begin{figure} \[ \begin{array}{c} \scalebox{0.75}{
\begin{tikzpicture}
    \node[shape=circle,draw=black] (S) at (0,0) {S(K)};
    \node[shape=circle,draw=black] (B) at (3,0) {B(K)};
    \node[shape=circle,draw=black] (Q) at (6,0) {Q(K)};
    \node[shape=circle,draw=black] (I) at (9,0) {I (K)};
    \node[shape=circle,draw=black] (V) at (4.5,4.5) {V(K)};
    \node[shape=circle,draw=black] (P) at (4.5,-4.5) {$\pi_1 (K)$}; 

    \draw [->] (V) edge node[left] {$\theta = I$} (B);
    \draw [->](V) edge node[left] {$\theta = \phi$}  (S);
    \path [->](V) edge node[right] {$\theta = \phi^{-1} $} (I);
    \path [->](V) edge node[right] {$\phi= I$} (Q);
    \path [->](B) edge node[left] {$ \phi = I$} (P);
    \path [->](S) edge node[left] {$ \theta = I$} (P);
    \path [->](I) edge node[right] {$ \theta = I$} (P);
    \path [->](Q) edge node[right] {$ \theta = I$} (P);
\end{tikzpicture}} \end{array} \]
\caption{Structure of $V(G)$ }
\label{fig:groupstruct} 
\end{figure}

\section{Parity Knot Groups} \label{paritygroups}
 
We use this structure to construct a knot group  that differentiates between odd and even crossings and retains sign information.
In a Reidemeister III move, there are either three even crossings or 2 odd crossings and 1 even crossing. There are three possible versions of the oriented Reidemeister III move when parity is taken into account. 
\begin{figure} 
\[ \begin{array}{c} \scalebox{0.75}{ 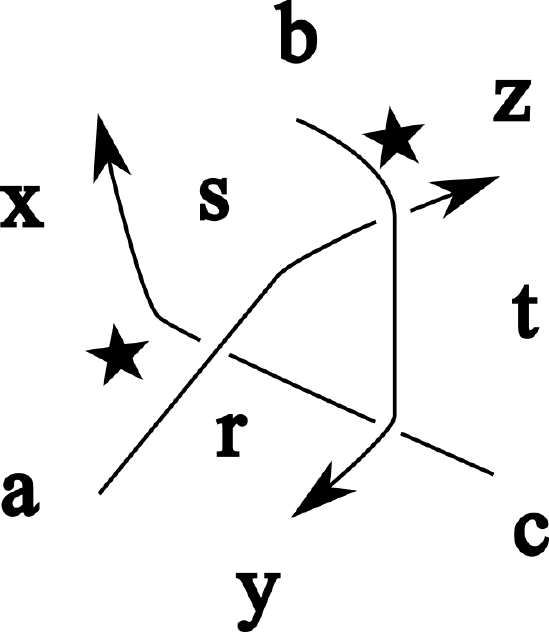} \end{array} \leftrightarrow 
\begin{array}{c} \scalebox{0.75}{ \input{diagrams/R3aRHScase1v2.pdf_tex}} \end{array}
 \]
\caption{Case 1} 
\label{fig:R3case1}
\end{figure}
In the diagrams, we have marked the odd crossings with a star.

To study the relations generally, we let $\mathcal{E}$ and $\mathcal{O}$ define $\theta$ for even and odd crossings respectively. 
We then use $\mathsf{e}$ and $\mathsf{o}$ to denote $\phi$ for even and odd crossings respectively. 
For simplicity, the bar symbol will be used to indicate inverses. For example,  $\bar{ \mathcal{O}} = \mathcal{O}^{-1} $.  We derive the relationships for each case. Note that a Reidemeister III contains only even crossings or two odd crossings and one even crossing. Hence, from the previous section,  $\mathcal{E}$ and $\mse$ commute.

In case 1, we refer to figure \ref{fig:R3case1}.
From the left hand side, we obtain:
\begin{align} \nonumber
s&=\mso (a), & x &=a \mathcal{O} (r s^{-1}), \\ \nonumber
t &= \mso (b), & z &= b \mathcal{O} (s t^{-1}), \\ \nonumber
y&= \mseb (t), &  r&= \mceb ( y^{-1} c) t.
\end{align}

From the right hand side, we obtain 
\begin{align}\nonumber
p&= \mseb (b), & x &= \mceb (p^{-1} q) b, \\  \nonumber
z&= \mso (m), & q &= m \mathcal{O} (c z^{-1}), \\  \nonumber
y&= \mso(p), & m&= p \mathcal{O} (a y^{-1}) .
\end{align}

This reduces to 
\begin{align} \label{eqn:c1x}
x &= a \mathcal{O} [ \mceb ( \mseb ( \mso (b^{-1}))    c)    \mso(ba^{-1})]
= \mceb \left[ \mathcal{O} [ a \mso( \mseb (b^{-1})   ) c  \mso \left[    \mathcal{O} ( \mso ( \mseb (b)) a^{-1} ) \mseb (b^{-1} )   \right]  \right] b , \\ \label{eqn:c1y}
y&= \mseb ( \mso (b))   = \mso( \mseb (b)) , \\ \label{eqn:c1z} 
z&= b \mathcal{O} ( \mso (a) \mso (b^{-1})) =    \mso \left(  \mseb (b) \mathcal{O} ( a \mso( \mseb (b)) ) \right)  .
\end{align}

\begin{figure} 
\[ \begin{array}{c} \scalebox{0.75}{ 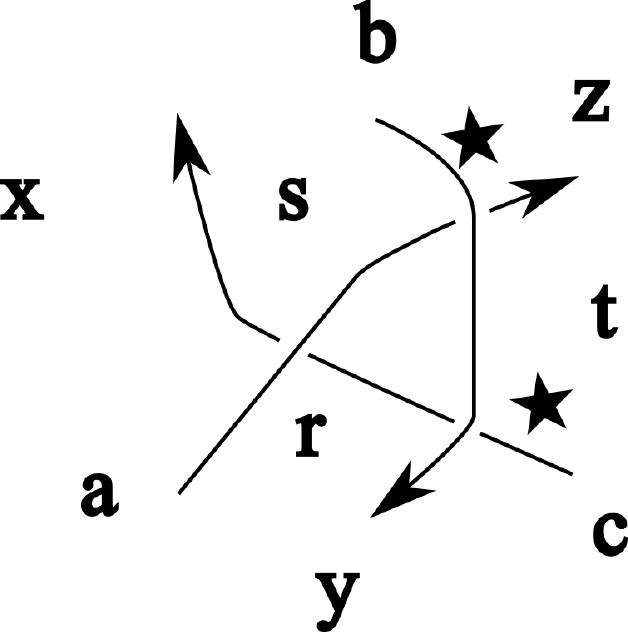} \end{array} \leftrightarrow 
\begin{array}{c} \scalebox{0.75}{ \input{diagrams/R3aRHScase2v2.pdf_tex}} \end{array}
 \]
\caption{Case 2} 
\label{fig:R3case2}
\end{figure}

In case 2, we refer to figure \ref{fig:R3case2}. 
From the left hand side, we obtain:
\begin{align} \nonumber
s&= \mse(a), & x &= a \mathcal{E} (r s^{-1}), \\ \nonumber
t&= \mso(b) ,& z& = b \mathcal{O} ( s t^{-1}), \\ \nonumber
y&= \msob (t), & r &= \mcob (b^{-1} c ) t .
\end{align}

From the right hand side, we obtain
\begin{align} \nonumber
p&=\msob (b), & x  &= \mcob (p^{-1} q) b,  \\ \nonumber
z&=\mse(m), & q&=m \mathcal{E} (c z^{-1}), \\ \nonumber
y&=\mso(p), & m&=p \mathcal{O} (a y^{-1}).
\end{align}

Reducing, we obtain
\begin{align*} 
x&= a \mathcal{E} ( \mathcal{O}^{-1} (b^{-1} c) \mso (b) \mse(a^{-1})) = ab \mathcal{O}^{-1} ( \mathcal{E} (c \mse(\mathcal{O} (b a^{-1}) \msob (b^{-1})) b ,\\ 
y&=b, \\ 
z&= b \mathcal{O}(\mse(a) \mso(b^{-1})) = \mse (\msob (b) \mathcal{O} (ab^{-1})) .
\end{align*}

\begin{figure} 
\[ \begin{array}{c} \scalebox{0.5}{ 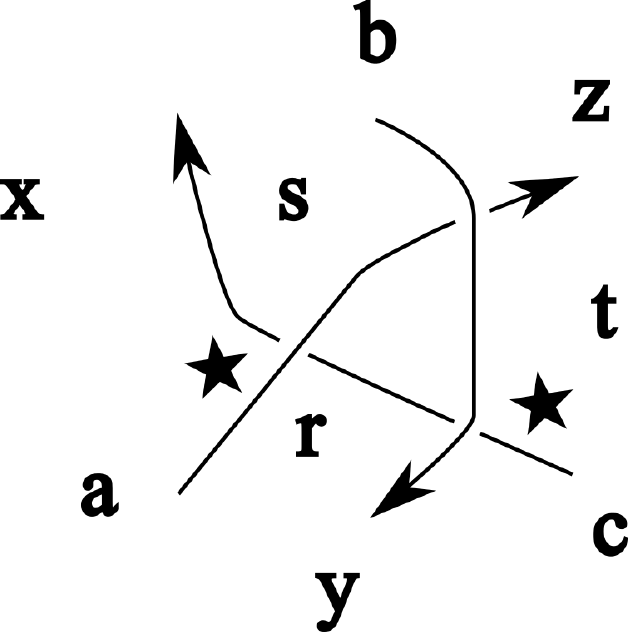} \end{array} \leftrightarrow 
\begin{array}{c} \scalebox{0.5}{ \input{diagrams/R3aRHScase3v2.pdf_tex}} \end{array}
 \]
\caption{Case 3} 
\label{fig:R3case3}
\end{figure}

The third case is in figure \ref{fig:R3case3}. From the left hand side,
\begin{align} \nonumber
s&=\mso(z), & x&= a \mathcal{O} (rs^{-1}),  \\ \nonumber
t&= \mse(b), & z &= b \mathcal{E} (st^{-1}), \\ \nonumber
y&= \msob (t), & r &= \mcob (y^{-1} c) t .
\end{align}

 From the right hand side
 \begin{align} \nonumber
 p&=\msob (b,) & x &= \mcob (p^{-1} q ) b, \\ \nonumber
 z&=\mso(m), & q&= m \mathcal{O}(cz^{-1}),  \\ \nonumber
 y&=\mse(p), & m&=p \mathcal{E}(a y^{-1}).
 \end{align}

This reduces to 
\begin{align} \label{eqn:c3x} 
x &= a b^{-1} c \mathcal{O} ( \mse(b) \mso(a^{-1})) =ab^{-1} \mceb \left[ \mathcal{O} ( c  \mso( \mathcal{E} (ba^{-1})) b^{-1}) ) \right] b,  \\  
y&=\msob ( \mse(b))= \mse( \msob (b)), \\  \label{eqn:c3z} 
z&= b \mathcal{E} (\mso(a) \mse(b^{-1})= b \mso (\mathcal{E} (ab^{-1})).
\end{align}
These relations must be satisfied to construct a parity knot group. 

We simplify some of the relations and determine that it is necessary for the  automorphisms to satisfy a set of commutator relations.

From equation \ref{eqn:c1x}, we  observe that
\begin{align} \nonumber
a \mathcal{O} [ \mceb ( \mseb ( \mso (b^{-1}))    c)    \mso(ba^{-1})]  &= \mceb \left[ \mathcal{O} [ a \mso( \mseb (b^{-1})   ) c  \mso \left[    \mathcal{O} ( \mso ( \mseb (b)) a^{-1} ) \mseb (b^{-1} )   \right]  \right] b.
\end{align}
Letting $a=b=1$, we obtain  $ \mathcal{O} ( \mceb (c))= \mceb ( \mathcal{O} (c)) $. Hence $ \mathcal{E} $ and $ \mathcal{O}$ commute. 

From equation \ref{eqn:c1y}, we observe that $\mseb ( \mso (b))= \mso(  \mseb (b))$, so $ \mse$ and $\mso$ commute. 

Finally, from equation \ref{eqn:c1z}
\begin{align} \nonumber
 b \mathcal{O} ( \mso (a) \mso (b^{-1})) &= \mso \left(  \mseb (b) \mathcal{O} ( a \mso( \mseb (b)) ) \right) . 
 \end{align}
Letting $b=1$, the equation reduces to  $ \mathcal{O} ( \mso (a)) = \mso ( \mathcal{O} (a))$ so that $ \mathcal{O}$ and $ \mso $ commute. 

In a group structure which preserves the parity and the crossing sign information, the automorphisms must have the following four commutator relations: $[ \mse, \mce ], [\mso, \mco], [\mco, \mce],$ and $[\mso, \mse]$.

\section{Restricting our maps} 

We prove that both parity and crossing information can not be detected within the specialized structures.

\subsection{B(K) parity groups}
In this section, we analyze the subsets from the group structure and let $\mathcal{E} = \mathcal{O} = I$. 

We consider  equation \ref{eqn:c3z}.  The simplified form of this equation
\begin{align} \label{eqn:biz}
b \mso(a) \mse(b^{-1})&=\mse( \msob (b)) .
\end{align}

From the simplified equation \ref{eqn:biz}, we  observe that $b \mso(a) \mse(b^{-1}) = b \mso(ab^{-1})$. As a result, $\mso(b) = \mse(b)$ for all $b$. Odd and even crossings are not distinguished.
\subsection{S(K) parity groups}

Here, $\mathcal{E}=\mse$ and $\mathcal{O}=\mso$.
Note that $\mathcal{E}$ and $ \mathcal{O}$ commute since in the previous section, we concluded that $\mso$ and $\mse$ commuted. 

We focus on equation \ref{eqn:c3z} and simplify: 
\begin{align} \label{eqn:siz}
b \mathcal{E}(\mathcal{O}(a)) \mathcal{E}^2 (b^{-1})  &= b \mathcal{O}(\mathcal{E}(ab^{-1})) .
\end{align} 
Since $\mathcal{O}$ and $\mathcal{E}$ commute, observe that $\mathcal{E}(b) = \mathcal{O}(b)$. 
This group type does not distinguish crossings by parity. 

\subsection{I(K) parity groups} 
In this parity group, we let
$\mathcal{O}=\msob $ and $ \mathcal{E} = \mseb $. 

We focus on equation \ref{eqn:c3z} and simplify:
\begin{align} \label{eqn:iiz}
 b \mathcal{E}(\mcob (a)) b^{-1}  &= b \mcob (\mathcal{E}(ab^{-1}))  .
\end{align}
Letting $a=1$, we obtain $b \mcob ( \mathcal{E} (b^{-1})) =1$ for all $b$.  Hence, $\mathcal{E} = \mathcal{O}$.  The odd and even crossings are not distinguished.

\subsection{Q(K) parity groups} 

Here $\mso=I=\mse$. 
We focus on equation \ref{eqn:c3x}. The expressions reduce to
\begin{align} \label{eqn:qix}
ab^{-1} c \mathcal{O}(ba^{-1})  &= ab^{-1} \mceb ( \mathcal{O} (c  \mathcal{E} (b a^{-1}) b^{-1})) b.
\end{align}
Letting $a=b=1$, we observe that $c = \mceb ( \mathcal{O} (c))$. Again, $ \mathcal{E} = \mathcal{O}$, so crossings are not distinguished. 

\begin{theorem}For a virtual knot $K$, the groups $B(K), I(K), S(K),$ and $Q(K)$ cannot detect both parity and crossing information. \end{theorem}
\begin{proof}See equations \ref{eqn:biz}, \ref{eqn:iiz}, \ref{eqn:siz}, and \ref{eqn:qix}. \end{proof}

\bibliographystyle{plain}

\end{document}